\documentclass{articlefederico}
\pdfoutput=1

\usepackage[ascii]{inputenc}
\usepackage[USenglish]{babel}
\usepackage{microtype}
\usepackage{amsmath}
\usepackage{amsfonts}
\usepackage{amsthm}
\usepackage{eucal}
\usepackage{tikz}
\usepackage{xypic}
\usepackage{aliascnt}
\usepackage[all,cmtip]{xy}
\usepackage[bookmarksnumbered]{hyperref}
\usepackage[all]{hypcap}

\title{\textsf{Gelfand--Tsetlin Polytopes and Feigin--Fourier--Littelmann Polytopes as\\Marked Poset Polytopes}}

\author{\textsf{Federico Ardila}\thanks{Supported in part by the National Science Foundation CAREER Award DMS-0956178 and the National Science Foundation Grant DMS-0801075.}
\and \textsf{Thomas Bliem}\thanks{Supported by the Deutsche Forschungsgemeinschaft Grant SPP 1388.}
\and \textsf{Dido Salazar}\thanks{Supported in part by the National Science Foundation Grant DGE-0841164.\newline
\emph{E-mail addresses:} \email{federico@math.sfsu.edu}, \email{bliem@math.sfsu.edu}, \email{dusalaza@sfsu.edu}.\newline
\emph{Postal address:} Department of Mathematics, San Francisco State University, 1600 Holloway Ave, San Francisco CA 94132, USA.}}

\hypersetup{pdfauthor={Federico Ardila, Thomas Bliem, and Dido Salazar},pdftitle={Gelfand-Tsetlin polytopes and Feigin-Fourier-Littelmann polytopes as marked poset polytopes}}

\date{}

\allowdisplaybreaks[4]

\newtheorem{theorem}{Theorem}[section]
\newaliascnt{question}{theorem}
\newtheorem{question}[question]{Question}
\aliascntresetthe{question}
\newaliascnt{proposition}{theorem}

\aliascntresetthe{proposition}
\newaliascnt{lemma}{theorem}

\aliascntresetthe{lemma}
\newaliascnt{corollary}{theorem}

\aliascntresetthe{corollary}

\theoremstyle{definition}

\newaliascnt{definition}{theorem}
\newtheorem{definition}[definition]{Definition}
\aliascntresetthe{definition}
\newaliascnt{example}{theorem}

\aliascntresetthe{example}
\newaliascnt{remark}{theorem}

\aliascntresetthe{remark}
\newaliascnt{conjecture}{theorem}
\newtheorem{conjecture}[conjecture]{Conjecture}
\aliascntresetthe{conjecture}

\def\Snospace~{\S{}}

\newcommand{\email}[1]{\href{mailto:#1}{\nolinkurl{#1}}}
\renewcommand{\epsilon}{\varepsilon}
\renewcommand{\phi}{\varphi}
\newcommand{\definedterm}{\textbf}
\newcommand{\suchthat}{\mid}
\DeclareMathOperator{\height}{ht}
\newcommand{\length}{\ell}

\begin{document}
\maketitle

\begin{abstract}
	\noindent 
	Stanley (1986) showed how a finite partially ordered set gives rise to two polytopes,  called the order polytope and chain polytope, which have the same Ehrhart polynomial despite being quite different combinatorially. We generalize his result to a wider family of polytopes constructed from a poset $P$ with integers assigned to some of its elements.
	
	Through this construction, we explain combinatorially the relationship between the Gelfand--Tsetlin polytopes (1950) and the Feigin--Fourier--Littelmann polytopes (2010), which arise in the representation theory of the special linear Lie algebra.
	We then use the generalized Gelfand--Tsetlin polytopes of Berenstein and Zelevinsky (1989) to propose conjectural analogues of the Feigin--Fourier--Littelmann polytopes corresponding to the symplectic and odd orthogonal Lie algebras.	
\end{abstract}

\section{\textsf{Introduction}}

	Consider the simple complex Lie algebra $\mathfrak{sl}_n$.
	The irreducible representations of $\mathfrak{sl}_n$ are parametrized up to isomorphism by dominant integral weights, i.e., weakly decreasing $n$-tuples of integers determined up to adding multiples of $(1, \ldots, 1)$.
	Given a dominant integral weight $\lambda$, let $V(\lambda)$ denote the corresponding irreducible $\mathfrak{sl}_n$-module.
	The module $V(\lambda)$ has a distinguished basis, the Gelfand--Tsetlin \cite{gelfand1950} basis, parametrized by the points with integral coordinates (``integral points'' or ``lattice points" for short) in the Gelfand--Tsetlin polytope $\mathrm{GT}(\lambda) \subset \mathbf{R}^{n(n-1)/2}$.

	Recently, Feigin, Fourier, and Littelmann \cite{feigin2010} constructed a different basis of $V(\lambda)$, related to the Poincar\'e--Birkhoff--Witt basis of the universal enveloping algebra $U(\mathfrak{n}^-)$, where $\mathfrak{n}^-$ is the span of the negative root spaces.
	Again, the basis elements are parametrized by the integral points in a certain polytope $\mathrm{FFL}(\lambda) \subset \mathbf{R}^{n(n-1)/2}$.
	
	Feigin, Fourier, and Littelmann used two subtle algebraic arguments to prove that their basis indeed spans $V(\lambda)$ and is linearly independent. When they had only produced the first half of the proof, they asked the second author of this paper:
	\begin{question} \label{question} \cite{fourier2010}
		Is there a combinatorial explanation for the fact that $\mathrm{GT}(\lambda)$ and $\mathrm{FFL}(\lambda)$ contain the same number of lattice points?
	\end{question}
	This question provided the motivation for this paper. We answer it by generalizing a result of Stanley \cite{Stanley86} on poset polytopes, as we now describe.
	Let $P$ be a finite poset. 
	Let $A$ be a subset of $P$ which contains all minimal and maximal elements of $P$. Let $\lambda = (\lambda_a)_{a \in A}$ be a vector in $\mathbf{R}^A$, which we think of as a marking of the elements of $A$ with real numbers. We call such a triple $(P,A,\lambda)$ a \definedterm{marked poset}. 

\begin{definition} The \definedterm{marked order polytope}  of $(P,A,\lambda)$ is
	\begin{align*}
		\mathcal{O}(P,A)_\lambda = \{x \in \mathbf{R}^{P-A} \suchthat
		{}& x_p \leq x_q \textrm{ for $p < q$,} \quad
		\lambda_a \leq x_p \textrm{ for $a < p$}, \\
		& x_p \leq \lambda_a \textrm{ for $p < a$}\}, 
	\end{align*}
	where $p$ and $q$ represent elements of $P-A$, and $a$ represents an element of $A$.
	The \definedterm{marked chain polytope}  of $(P,A,\lambda)$ is
	\begin{align*}
		\mathcal{C}(P,A)_\lambda = \{x \in \mathbf{R}_{\geq 0}^{P-A} \suchthat
		{}& x_{p_1} + \cdots + x_{p_k} \le \lambda_b-\lambda_a \\
		& \textrm{for $a < p_1 < \cdots < p_k < b$}\},
	\end{align*}	
where $a,b$ represent elements of $A$, and $p_1, \ldots, p_k$ represent elements of $P-A$.
\end{definition}

	For any polytope with integer coordinates $Q$ there exists a polynomial $E_Q(t)$, the \definedterm{Ehrhart polynomial} of $Q$, with the following property: for every positive integer $n$, the $n$-th dilate $nQ$ of $Q$ contains exactly $E_Q(n)$ lattice points (see \cite{Stanleybook}).
	With this notion, our answer to \autoref{question} is given by the following two results.

	\begin{theorem}\label{maintheorem}
		For any marked poset $(P,A, \lambda)$ with $\lambda \in \mathbf{Z}^A$, the marked order polytope $\mathcal{O}(P,A)_\lambda$ and the marked chain polytope $\mathcal{C}(P,A)_\lambda$  have the same Ehrhart polynomial.
	\end{theorem}

	\begin{theorem}
		For every partition $\lambda$ there exists a marked poset $(P,A,\lambda)$ such that $\mathrm{GT}(\lambda) = \mathcal{O}(P,A)_\lambda$ and $\mathrm{FFL}(\lambda) = \mathcal{C}(P,A)_\lambda$.
	\end{theorem}

	We also consider the extension of these constructions to other Lie algebras. Berenstein and Zelevinsky proposed a construction of generalized Gelfand--Tsetlin polytopes  \cite{berenstein1989} for other semisimple Lie algebras. For the symplectic and odd orthogonal Lie algebras, their polytopes are also in the family of marked order polytopes. Therefore \autoref{maintheorem} yields candidates for the Feigin--Fourier--Littelmann polytopes in types $B_n$ and $C_n$.

	The paper is organized as follows.
	In \autoref{F6Z6kCUH} we discuss the relevant aspects of the representation theory of the simple complex Lie algebras $\mathfrak{sl}_n$.
	Section \ref{VHvEjfL2} treats marked order and chain polytopes, and gives a bijection between their lattice points.
	Section \ref{sec:applications} discusses the application of the combinatorial results of \autoref{VHvEjfL2} to the representation theoretic polytopes that interest us. 

	We note that the combinatorial \autoref{VHvEjfL2} is self-contained, and may be of independent interest beyond the representation theoretic application. A possible way to read this article is to skip \autoref{F6Z6kCUH} and continue there directly.
	
\section{\textsf{Preliminaries}}
\label{F6Z6kCUH}

	Consider the simple complex Lie algebra $\mathfrak{sl}_n$.
Let $\mathfrak{h}$ be the Cartan subalgebra consisting of its diagonal matrices.
	For $i = 1, \ldots, n$, let $\epsilon_i \in \mathfrak{h}^*$ denote the projection onto the $i$-th diagonal component.
	As $\epsilon_1 + \cdots + \epsilon_n = 0$, the coefficient vector of an integral weight is only determined as an element of $\mathbf{Z}^n/\langle(1, \ldots, 1)\rangle$.
	We identify an integral weight with the corresponding equivalence class of coefficient vectors.
	If $\lambda$ is a weight and we use the symbol $\lambda$ in a context where it has to be interpreted as an $n$-tuple $\lambda = (\lambda_1, \ldots, \lambda_n)$, we use the convention that a representative has been chosen implicitly.
	Fix simple roots $\alpha_i = \epsilon_i - \epsilon_{i+1}$ for $i = 1, \ldots, n-1$.
	The corresponding fundamental weights are $\omega_i = \epsilon_1 + \cdots + \epsilon_i$.
	Hence dominant integral weights correspond to weakly decreasing $n$-tuples of integers, or partitions.	

	Given a dominant integral weight $\lambda$, the associated Gelfand--Tsetlin \cite{gelfand1950} polytope $\mathrm{GT}(\lambda)$ is defined as follows:
	Consider the board given in \autoref{zczmmXWC}.	
	
	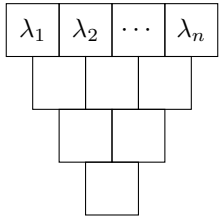
\begin{figure}[h]
		\centering
		\begin{tikzpicture}[x=2em,y=-2em]
			\draw (0,0) rectangle +(1,1);
			\draw (1,0) rectangle +(1,1);
			\draw (2,0) rectangle +(1,1);
			\draw (3,0) rectangle +(1,1);
			\draw (0.5,1) rectangle +(1,1);
			\draw (1.5,1) rectangle +(1,1);
			\draw (2.5,1) rectangle +(1,1);
			\draw (1,2) rectangle +(1,1);
			\draw (2,2) rectangle +(1,1);
			\draw (1.5,3) rectangle +(1,1);
			\draw (0.5,0.5) node{$\lambda_1$};
			\draw (1.5,0.5) node{$\lambda_2$};
			\draw (2.5,0.5) node{$\cdots$};
			\draw (3.5,0.5) node{$\lambda_n$};
		\end{tikzpicture}
		\caption{Board defining Gelfand--Tsetlin patterns.}
		\label{zczmmXWC}
	\end{figure}
	
	Each one of the $n(n-1)/2$ empty boxes stands for a real variable. The polytope $\mathrm{GT}(\lambda) \subset \mathbf{R}^{n(n-1)/2}$ is given by the fillings of the board with real numbers with the following property:  each number is less than or equal to its upper left neighbor and greater than or equal to its upper right neighbor.
	Note that the ambiguity in choosing an $n$-tuple for the weight $\lambda$ amounts to an integral translation of $\mathrm{GT}(\lambda)$, and hence does not affect its number of integral points.
	In fact, the integral points in $\mathrm{GT}(\lambda)$ parametrize the Gelfand--Tsetlin basis of $V(\lambda)$, hence $\lvert \mathrm{GT}(\lambda) \cap \mathbf{Z}^{n(n-1)/2} \rvert = \dim V(\lambda)$.

	Feigin, Fourier, and Littelmann \cite{feigin2010} associate a different polytope with a dominant integral weight $\lambda$ as follows:
	The positive roots of $\mathfrak{sl}_n$ are $\Phi_+ = \{ \alpha_{i,j} \suchthat 0 \leq i < j \leq n \}$, where $\alpha_{i,j} = \epsilon_i - \epsilon_j$.
	A Dyck path is by definition a sequence $(\beta(0), \ldots, \beta(k))$ in $\Phi_+$ such that $\beta(0)$ and $\beta(k)$ are simple, and if $\beta(l) = \alpha_{i,j}$, then either $\beta(l+1) = \alpha_{i+1,j}$ or $\beta(l+1) = \alpha_{i,j+1}$.
	Denote the coordinates on $\mathbf{R}^{\Phi_+}$ by $s_\beta$ for $\beta \in \Phi_+$.
	Let $\lambda = m_1 \omega_1 + \cdots + m_{n-1} \omega_{n-1}$.
	Then the polytope $\mathrm{FFL}(\lambda) \subset \mathbf{R}^{\Phi_+}$ is given by the inequalities \[ s_\beta \geq 0 \] for all $\beta \in \Phi_+$ and \[ s_{\beta(0)} + \cdots + s_{\beta(k)} \leq m_i + \cdots + m_j \] for all Dyck paths $(\beta(0), \ldots, \beta(k))$ such that $\beta(0) = \alpha_i$ and $\beta(k) = \alpha_j$.

	For all $\alpha \in \Phi_+$, let $f_\alpha$ be a nonzero element of the root space $\mathfrak{g}_{-\alpha}$.
	Let $v_\lambda$ be a highest weight vector of $V(\lambda)$.
	Fix any total order on $\Phi_+$. As $s$ ranges over the lattice points of $\mathrm{FFL}(\lambda)$, the elements $\bigl( \prod_{\alpha \in \Phi_+} f_\alpha^{s_\alpha} \bigr) v_\lambda$ form a basis of $V(\lambda)$ \cite[Th.\ 3.11]{feigin2010}.
	Hence $\lvert \mathrm{FFL}(\lambda) \cap \mathbf{Z}^{\Phi_+} \rvert = \dim V(\lambda)$.
	
	The previous discussion shows that $\lvert \mathrm{FFL}(\lambda) \cap \mathbf{Z}^{\Phi_+} \rvert = \lvert \mathrm{GT}(\lambda) \cap \mathbf{Z}^{n(n-1)/2} \rvert$. In the sequel, we give a combinatorial explanation and an extension of this fact.

\section{\textsf{Marked poset polytopes}}
\label{VHvEjfL2}

	To any finite poset $P$, Stanley \cite{Stanley86} associated two polytopes in $\mathbf{R}^P$: the order polytope and the chain polytope.
	He showed that there is a continuous, piecewise linear bijection between them, which restricts to a bijection between their sets of integral points.
	In this section we construct a generalization of the order and chain polytopes, and prove  the analogous result. 
	We begin with a review of Stanley's work.

\subsection{\textsf{Stanley's order and chain polytopes}}

	Let $P$ be a finite poset. For $p, q \in P$ we say that $p$ \definedterm{covers} $q$, and write $p \succ q$, when $p>q$ and there is no $r \in P$ with $p>r>q$. We identify $P$ with its \definedterm{Hasse diagram}: the graph with vertex set $P$, having an edge going down from $p$ to $q$ whenever $p$ covers $q$.
	
	The \definedterm{order polytope} and \definedterm{chain polytope} of $P$ are,
	\begin{align*}
		\mathcal{O}(P)
		&= \{x \in [0,1]^P \suchthat x_p \leq x_q \textrm{ for all $p < q$}\},\textrm{ and} \\
		\mathcal{C}(P)
		&= \{x \in [0,1]^P \suchthat x_{p_1} + \cdots + x_{p_k} \le 1 \textrm{ for all chains $p_1 < \cdots < p_k$}\}.
	\end{align*}
respectively.

	Stanley proved that, even though $\mathcal{O}(P)$ and $\mathcal{C}(P)$ can have quite different combinatorial structures, they have the same Ehrhart polynomial.	
	He did this as follows. Define the \definedterm{transfer map} $\phi : \mathbf{R}^P \to \mathbf{R}^P$ by
	\begin{equation} \label{uVhsMXQ7}
		\phi(x)_p = \begin{cases}
			x_p & \text{if $p$ is minimal,} \\
			\min{} \{ x_p - x_q \suchthat p \succ q \} & \text{otherwise}
		\end{cases}
	\end{equation}
	for $x \in \mathbf{R}^P$, $p \in P$. Then:
	
	\begin{theorem}[{\cite[Theorem 3.2]{Stanley86}}] \label{stanley's-theorem}
		The transfer map $\phi$ restricts to a continuous, piecewise linear bijection from $\mathcal{O}(P)$ onto $\mathcal{C}(P)$.
		For any $m \in \mathbf{N}$, $\phi$ restricts to a bijection from $\mathcal{O}(P) \cap \frac 1m \mathbf{Z}^P$ onto $\mathcal{C}(P) \cap \frac 1m \mathbf{Z}^P$.
	\end{theorem}

\subsection{\textsf{Marked poset polytopes}}	

	We now recall the definition of marked order and chain polytopes, and prove that they satisfy a generalization of \autoref{stanley's-theorem}.

	An element of a poset is called \definedterm{extremal} if it is maximal or minimal.
	
	\begin{definition} \label{vyCu2e3F}
		A \definedterm{marked poset} $(P,A,\lambda)$ consists of a finite poset $P$, a subset $A \subseteq P$ containing all its extremal elements, and a vector $\lambda \in \mathbf{R}^A$.
		We identify it with the \definedterm{marked Hasse diagram}, where we label the elements $a \in A$ with $\lambda_a$ in the Hasse diagram of $P$.
	\end{definition}
	
	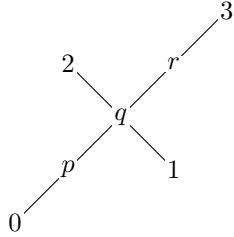
\begin{figure}[h]
		\centering
		\begin{tikzpicture}[x=2em,y=2em]
			\draw (0,0) -- (4,4);
			\draw (3,1) -- (1,3);
			\draw (0,0) node[fill=white,inner sep=0.2ex]{$0$};
			\draw (1,1) node[fill=white,inner sep=0.2ex]{$p$};
			\draw (2,2) node[fill=white,inner sep=0.2ex]{$q$};
			\draw (3,3) node[fill=white,inner sep=0.2ex]{$r$};
			\draw (4,4) node[fill=white,inner sep=0.2ex]{$3$};
			\draw (3,1) node[fill=white,inner sep=0.2ex]{$1$};
			\draw (1,3) node[fill=white,inner sep=0.2ex]{$2$};
		\end{tikzpicture}
		\caption{A marked Hasse diagram defining a partial order on the set $P = \{p, q, r\} \cup A$ with $\lvert A \rvert = 4$ and $\lambda = (3,2,1,0) \in \mathbf{R}^A$.}
		\label{fig:marked-hasse-diagram}
	\end{figure}

	\begin{definition}
		The \definedterm{marked order polytope}  of $(P,A,\lambda)$ is
		\begin{align*}
			\mathcal{O}(P,A)_\lambda
			= \{x \in \mathbf{R}^{P-A} \suchthat
			{}& x_p \leq x_q \textrm{ for $p < q$,} \\
			& \lambda_a \leq x_p \textrm{ for $a < p$}, \\
			& x_p \leq \lambda_a \textrm{ for $p < a$}\}, 
		\end{align*}
		where $p$ and $q$ represent elements of $P-A$, and $a$ represents an element of $A$.
		The \definedterm{marked chain polytope}  of $(P,A,\lambda)$ is
		\begin{align*}
			\mathcal{C}(P,A)_\lambda
			= \{x \in \mathbf{R}_{\geq 0}^{P-A} \suchthat
			{}& x_{p_1} + \cdots + x_{p_k} \le \lambda_b-\lambda_a \\
			& \textrm{for $a < p_1 < \cdots < p_k < b$}\},
		\end{align*}	
		where $a,b$ represent elements of $A$, and $p_1, \ldots, p_k$ represent elements of $P-A$.
	\end{definition}

	\begin{figure}[h]
		\centering
		\begin{tikzpicture}[x=(-120:5em),y=(0:7em),z=(90:7em)]
			\draw[->] (0,0,0) -- (1.1,0,0) node[below]{$x_p$};
			\draw[->] (0,0,0) -- (0,1.1,0) node[right]{$x_q$};
			\draw[->] (0,0,0) -- (0,0,1.1) node[above]{$x_r$};
			\draw (1,0,0) -- (1,-.05,0) node[left]{$3$};
			\draw (0,1,0) -- (0,1,-.05) node[below]{$3$};
			\draw (0,0,1) -- (0,-.05,1) node[left]{$3$};
			\draw (0,0.33,1) -- (0.33,0.33,1) -- (0.33,0.33,0.33) -- (0,0.33,0.33) -- cycle;
			\draw (0,0.67,1) -- (0.67,0.67,1) -- (0.67,0.67,0.67) -- (0,0.67,0.67) -- cycle;
			\draw (0,0.33,1) -- (0,0.67,1);
			\draw (0.33,0.33,1) -- (0.67,0.67,1);
			\draw (0.33,0.33,0.33) -- (0.67,0.67,0.67);
			\draw (0,0.33,0.33) -- (0,0.67,0.67);
			\draw[dashed] (0,0,1) -- (1,0,1) -- (1,1,1) -- (1,1,0) -- (1,0,0) -- (1,0,1);
			\draw[dashed] (1,1,0) -- (0,1,0) -- (0,1,1);
			\draw[dashed] (0,0,1) -- (0,0.33,1);
			\draw[dashed] (0,1,1) -- (0,0.67,1);
			\draw[dashed] (0,1,1) -- (1,1,1);
			\draw[dashed] (0,0,0) -- (0.33,0.33,0.33);
			\draw[dashed] (1,1,1) -- (0.67,0.67,0.67);
			\draw[dashed] (0,0,0) -- (0,0.33,0.33);
			\draw[dashed] (0,1,1) -- (0,0.67,0.67);
			\draw[dashed] (0,0,1) -- (0.33,0.33,1);
			\draw[dashed] (1,1,1) -- (0.67,0.67,1);
		\end{tikzpicture}
		\quad
		\begin{tikzpicture}[x=(-120:5em),y=(0:7em),z=(90:7em)]
			\draw[->] (0,0,0) -- (1.1,0,0) node[below]{$x_p$};
			\draw[->] (0,0,0) -- (0,1.1,0) node[right]{$x_q$};
			\draw[->] (0,0,0) -- (0,0,1.1) node[above]{$x_r$};
			\draw (1,0,0) -- (1,-.05,0) node[left]{$3$};
			\draw (0,1,0) -- (0,1,-.05) node[below]{$3$};
			\draw (0,0,1) -- (0,-.05,1) node[left]{$3$};
			\draw (0,0,0.67) -- (0.33,0,0.67) -- (0.67,0,0.33) -- (0.67,0,0) -- (0.33,0.33,0) -- (0,0.33,0) -- (0,0.33,0.33) -- cycle;
			\draw (0.67,0,0.33) -- (0.33,0.33,0.33) -- (0.33,0.33,0);
			\draw (0.33,0,0.67) -- (0.33,0.33,0.33) -- (0,0.33,0.33);
			\draw[dashed] (0,1,0) -- (0,1,1) -- (1,1,1) -- (1,0,1) -- (1,0,0);
			\draw[dashed] (0,1,1) -- (0,0,1) -- (1,0,1);
			\draw[dashed] (1,0,0) -- (1,1,0) -- (0,1,0);
			\draw[dashed] (1,1,0) -- (1,1,1);
		\end{tikzpicture}
		\caption{The marked order polytope of the marked poset in \autoref{fig:marked-hasse-diagram} is given by the inequalities $0 \leq x_p \leq x_q \leq x_r \leq 3$ and $1 \leq x_q \leq 2$. The marked chain polytope is given by the inequalities $x_p, x_q, x_r \geq 0$, $x_p + x_q + x_r \leq 3$, $x_p + x_q \leq 2$, $x_q + x_r \leq 2$, and $x_q \leq 1$. Note that they are not combinatorially isomorphic.}
		\label{fig:marked-order-polytope}
	\end{figure}
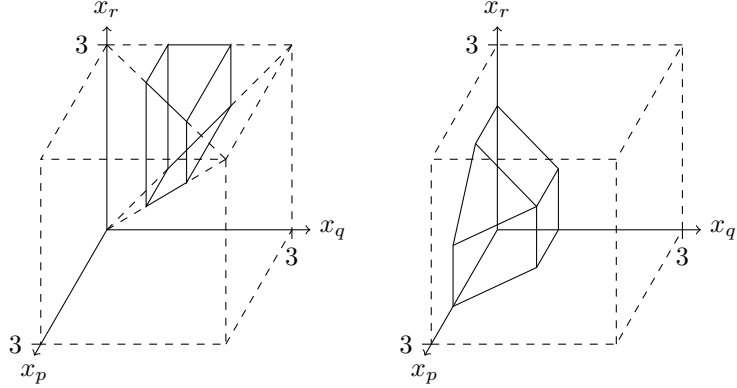
	
	Stanley's construction is a special case of ours  as follows:
	Given any finite poset $P$, add a new smallest and largest element to obtain $\tilde P = P \cup \{ \hat 0, \hat 1 \}$ for $\hat 0, \hat 1 \notin P$.
	Let $A = \{ \hat 0, \hat 1 \}$ and $\lambda = (0, 1)$.
	Then
	\[
		\mathcal{O}(P) = \mathcal{O}(\tilde P, A)_{\lambda}
		\quad \text{and} \quad
		\mathcal{C}(P) = \mathcal{C}(\tilde P, A)_{\lambda} .
	\]
	
	The following definitions will be needed in the proof of \autoref{lem:genPhi}:
	The \definedterm{length} of a chain $C=\{p_1 < \cdots < p_k\} \subseteq P$ is $\length(C) = k - 1$.
	The \definedterm{height} of $p \in P$ is the length of the longest chain ending at $p$.
	If $P$ is graded, the height of an element is just its rank.
	
	\begin{theorem} \label{lem:genPhi}
		Let $(P,A,\lambda)$ be a marked poset. The map $\tilde\phi: \mathbf{R}^{P-A} \to \mathbf{R}^{P-A}$ defined by 
		\[
			\tilde\phi(x)_p
			=	\min{} \left(\{ x_p - x_q \suchthat p \succ q, q \notin A\} \cup \{x_p - \lambda_q \suchthat p \succ q, q \in A \}\right) 
		\]
		for each $p \in P-A$ restricts to a continuous, piecewise affine bijection from $\mathcal{O}(P,A)_\lambda$ onto $\mathcal{C}(P,A)_\lambda$.
	\end{theorem}

	The following alternative description of $\tilde\phi$ may be useful. 
	Let $\phi : \mathbf{R}^P \to \mathbf{R}^P$ be Stanley's transfer map as defined in \eqref{uVhsMXQ7}.
	Let $\pi: \mathbf{R}^P \to \mathbf{R}^{P-A}$ be the canonical projection which forgets the coordinates in $A$, and let $i : \mathbf{R}^{P-A} \to \mathbf{R}^P$ be the canonical inclusion into the fiber over $\lambda \in \mathbf{R}^A$, which adds a coordinate $\lambda_a$ to each $a \in A$.
	Then $\tilde\phi = \pi \circ \phi \circ i$.

	These maps (and some more to be defined in the proof) are illustrated in the following diagram.
			
\begin{large}
\[
\xymatrixrowsep{6pc}
\xymatrixcolsep{3.5pc}
	\xymatrix{ 
	\mathbf{R}^P \ar[r]^{\phi}&  \mathbf{R}^P \ar[d]^\pi \\
	\mathcal{O}(P,A)_\lambda\,\, \ar@/^/[r]^{\tilde\phi} \ar@{^{(}->}[u]^i &  \,\,\mathcal{C}(P,A)_\lambda  \ar@/^/[l]^{\tilde\psi} \ar[lu]^{\psi}
	}
\]
\end{large}
	
	\begin{proof}
		We start by showing that $\tilde\phi(\mathcal{O}(P,A)_\lambda) \subseteq \mathcal{C}(P,A)_\lambda$.
		Let $x \in \mathcal{O}(P, A)_\lambda$ and $y = \tilde\phi(x)$.
		Let $a, b \in A$, and $p_1, \ldots, p_k \in P-A$ be such that $a < p_1 < \cdots < p_k < b$.
		The definition of $\phi$ implies that $y_{p_i} \le x_{p_i} - x_{p_{i-1}}$ for all $i = 2, \ldots, k$ and $y_{p_1} \leq x_{p_1} - \lambda_a$.
		Thus,
		\begin{align*}
			y_{p_1} +  \cdots + y_{p_k}
			&\le (x_{p_1} - \lambda_a) + (x_{p_2} - x_{p_1}) + \cdots + (x_{p_k} - x_{p_{k-1}}) \\
			&= x_{p_k} - \lambda_a
			\leq \lambda_b - \lambda_a .
		\end{align*}
		Hence, $y \in \mathcal{C}(P,A)_\lambda$.
	
		To show that $\tilde\phi$ is bijective, we construct its inverse $\tilde\psi: \mathcal{C}(P, A) \to \mathcal{O}(P, A)$.
		We first define a map $\psi: \mathbf{R}^{P-A} \to \mathbf{R}^{P}$, where we define $\psi(y)_p$ recursively by going up the poset according to the rule:
		\[
			\psi(y)_p =  \begin{cases}
				\lambda_p & \text{if }  p \in A, \\ 
				y_p + \max{} \{\psi(y)_q \suchthat p \succ q \} & \text{if } p \notin A.
			\end{cases}
		\]
		Since all the elements of height $0$ are in $A$, $\psi(y)$ is well-defined. We then define $\tilde\psi = \pi \circ \psi$ by applying $\psi$ and then forgetting the $A$-coordinates.
		We will prove that, when restricted to $\mathcal{C}(P, A)_\lambda$, the map $\tilde\psi$ is the inverse of $\tilde\phi$.
		
		First we show that $\tilde\psi \circ \tilde\phi$ is the identity on $\mathcal{O}(P, A)_\lambda$.
		We begin by showing that  $\psi \circ \tilde\phi = i$; \emph{i.e.}, that if $x \in \mathcal{O}(P, A)_\lambda$ and $y = \tilde\phi(x)$ then $i(x) = \psi(y)$. We prove $i(x)_p = \psi(y)_p$ by induction on $\height(p)$.
		The claim certainly holds for $\height(p) = 0$.
		Suppose that we have proved it for all elements of height at most $n$, and let $p$ have height $n+1$. 
		If $p \in A$, then
		\[
			\psi(y)_p = \lambda_p = i(x)_p
		\]
		by definition. Otherwise, if $p \notin A$, we have 
		\begin{align*}
			\psi(y)_p
			&= y_p + \max{} \{\psi(y)_q \suchthat p \succ q \} \\
			&= y_p + \max{} \{i(x)_q \suchthat p \succ q \} .
		\end{align*}
		by the inductive hypothesis.
		As
		\begin{align*}
			y_p
			&= \tilde\phi(x)_p
			= \pi(\phi(i(x)))_p
			= \phi(i(x))_p \\
			&= \min{} \{ i(x)_p - i(x)_q \suchthat p \succ q \} \\
			&= i(x)_p - \max{} \{ i(x)_q \suchthat p \succ q \},
		\end{align*}
		we conclude that $\psi(y)_p = i(x)_p$, as desired.

		We have shown that $\psi \circ \tilde\phi = i$. By composing with the projection which forgets the $A$ coordinates, we obtain that $\tilde\psi \circ \tilde\phi$ is the identity on $\mathcal{O}(P, A)_\lambda$.
		Hence $\tilde\phi$ is injective.

		To prove surjectivity, let $y \in \mathcal{C}(P,A)_\lambda$ and define $x = \tilde\psi(y) \in \mathbf{R}^{P-A}$.
		We start by showing that $x \in \mathcal{O}(P,A)_\lambda$.
		Let $p \in P - A$.
		By definition,
		\[
			x_p
			= \psi(y)_p
			= y_p + \max{} \{\psi(y)_q \suchthat p \succ q \}
		\]
		As $y_p \geq 0$, this implies $x_p \geq \psi(y)_q$ for all $q$ such that $p \succ q$. If $q \in A$, this says that $x_p \geq \lambda_q$. If $q \notin A$, this says that $x_p \geq x_q$.
		As $p$ is arbitrary, it follows that $x \in \mathcal{O}(P,A)_\lambda$. 
		
		Finally, we claim that $\tilde\phi(x)=y$. Once again, we prove that $\tilde\phi(x)_p = y_p$ for all $p \in P-A$ by induction on the height of $p$.
		For height $0$ this statement is vacuous.
		Suppose that it holds for all elements of height at most $n$, and consider $p \in P - A$ with $\height(p) = n+1$.
		Then
		\begin{align*}
			\tilde\phi(x)_p
			&= \min{} \{ i(x)_p - i(x)_q \suchthat p \succ q \} \\
			&= \min{} \{ \psi(y)_p - \psi(y)_q \suchthat p \succ q \} \\
			&= \psi(y)_p - \max{} \{ \psi(y)_q \suchthat p \succ q \} \\
			&= y_p + \max{} \{ \psi(y)_q \suchthat p \succ q \} - \max{} \{ \psi(y)_q \suchthat p \succ q \} \\
			&= y_p,
		\end{align*}
		as desired.
		We have shown that $\tilde\phi \circ \tilde\psi$ is the identity on $\mathcal{C}(P,A)_\lambda$, hence $\tilde\phi$ is surjective.
		
		We conclude that $\tilde\psi: \mathcal{C}(P,A)_\lambda \to \mathcal{O}(P,A)_\lambda$ and $\tilde\phi: \mathcal{O}(P,A)_\lambda \to \mathcal{C}(P,A)_\lambda$ are inverse functions, and therefore bijective, as we wished to show. 
		The fact that they are continuous and piecewise affine follows directly from the definitions.
	\end{proof}

	We conclude this section with the generalization of the second part of \autoref{stanley's-theorem}, the compatibility of the transfer map with the integral lattice.
	
	\begin{theorem} \label{thm:PhiLatticePoint}
		Let $(P,A, \lambda)$ be a marked poset with $\lambda \in \mathbf{Z}^A$.
		Then $\tilde\phi$ restricts to a bijection between $\mathcal{O}(P, A)_\lambda \cap \frac1m \mathbf{Z}^{P-A}$ and $\mathcal{C}(P, A)_\lambda \cap \frac1m \mathbf{Z}^{P-A}$.
		Therefore $\mathcal{O}(P, A)_\lambda$ and $\mathcal{C}(P, A)_\lambda$ have the same Ehrhart polynomial.
	\end{theorem}
	
	\begin{proof}
		This follows immediately from the proof of \autoref{lem:genPhi}, as both $\tilde\phi$ and $\tilde\psi$ preserve integrality.
	\end{proof}

	It is worth noting that \autoref{thm:PhiLatticePoint} does not hold for general $\lambda \in \mathbf{R}^A$.

\section{\textsf{Applications}}
\label{sec:applications}

	We now show how marked poset polytopes occur ``in nature" in the representation theory of semisimple Lie algebras. More concretely, marked order polytopes occur as Gelfand--Tsetlin polytopes in type $A$, $B$, and $C$, and marked chain polytopes occur as Feigin--Fourier--Littelmann polytopes in type $A$.

\subsection{\textsf{Type $A$.}}

	Let $\lambda$ be a dominant integral weight for $\mathfrak{sl}_n$.
	Let $\mathcal{O}(P,A)_\lambda$ and $\mathcal{C}(P,A)_\lambda$ be the marked order and chain polytopes determined by the marked Hasse diagram given in \autoref{neDyVEAe}.
	\begin{figure}[h]
		\centering
		\begin{tikzpicture}[x=(45:2em),y=(-45:2em)]
			\draw (0,0) -- (3,0);
			\draw (0,1) -- (2,1);
			\draw (0,2) -- (1,2);
			\draw (0,3) -- (0,3);
			\draw (0,0) -- (0,3);
			\draw (1,0) -- (1,2);
			\draw (2,0) -- (2,1);
			\filldraw (0,0) circle(0.4ex);
			\filldraw (1,0) circle(0.4ex);
			\filldraw (2,0) circle(0.4ex);
			\filldraw (0,1) circle(0.4ex);
			\filldraw (1,1) circle(0.4ex);
			\filldraw (0,2) circle(0.4ex);
			\draw (3,0) node[fill=white,inner sep=0.2ex]{$\lambda_1$};
			\draw (2,1) node[fill=white,inner sep=0.2ex]{$\lambda_2$};
			\draw (1,2) node[fill=white,inner sep=0.2ex,text height=2ex]{$\vdots$};
			\draw (0,3) node[fill=white,inner sep=0.2ex]{$\lambda_n$};

		\end{tikzpicture}
		\caption{Marked Hasse diagram for $\mathfrak{sl}_n$.}
		\label{neDyVEAe}
	\end{figure}
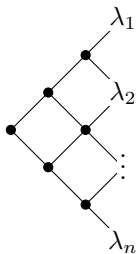
	Note that \autoref{neDyVEAe} is obtained from \autoref{zczmmXWC} by a clockwise rotation by $90^\circ$.
	Hence from the definitions it is immediate that $\mathrm{GT}(\lambda) = \mathcal{O}(P, A)_\lambda$.	
	Similarly, it follows immediately from the definitions that $\mathrm{FFL}(\lambda) = \mathcal{C}(P, A)_\lambda$.
	Hence the equation
	\[
		\lvert \mathrm{FFL}(\lambda) \cap \mathbf{Z}^{\Phi_+} \rvert = \lvert \mathrm{GT}(\lambda) \cap \mathbf{Z}^{n(n-1)/2} \rvert
	\]
	is implied by \autoref{thm:PhiLatticePoint}.
	
	It would be interesting to see whether the explicit bijection of \autoref{thm:PhiLatticePoint} gives interesting information about the transition matrix between the Gelfand--Tsetlin basis and the Feigin--Fourier--Littelmann basis of $V(\lambda)$.

\subsection{\textsf{Type $C$.}}

	Now consider the symplectic Lie algebra $\mathfrak{sp}_{2n}$.
	Here the role of Gelfand--Tsetlin patterns is played by the generalized Gelfand--Tsetlin patterns defined by Berenstein and Zelevinsky \cite{berenstein1989}.
	Fix a Cartan subalgebra $\mathfrak{h} \subset \mathfrak{sp}_{2n}$.
	Choose simple roots $\alpha_1, \ldots, \alpha_n \in \mathfrak{h}^*$ such that $\alpha_i \not\perp \alpha_{i + 1}$ for $i < n$ and $\alpha_n$ is the long root.
	Let $\epsilon_1, \ldots, \epsilon_n$ be the basis of $\mathfrak{h}^*$ such that $\alpha_i = \epsilon_i - \epsilon_{i+1}$ for $i < n$ and $\alpha_n = 2\epsilon_n$.
	The corresponding fundamental weights are $\omega_i = \epsilon_1 + \cdots + \epsilon_i$.
	This is the setting as used by Bourbaki \cite{bourbaki1981}.
	We identify a weight $\lambda$ with the $n$-tuple $(\lambda_1, \ldots, \lambda_n)$ of its coefficients with respect to the basis $\epsilon_1, \ldots, \epsilon_n$.
	Then dominant integral weights correspond to weakly decreasing $n$-tuples of nonnegative integers.
	Given a dominant integral weight $\lambda$, Berenstein and Zelevinsky define an \definedterm{$\mathfrak{sp}_{2n}$-pattern} of highest weight $\lambda$ to be a filling of the board in \autoref{ubf7JTGE} with nonnegative integers, such that every number is bounded from above by its upper left neighbor and bounded from below by its upper right neighbor (if any).
	\begin{figure}[h]
		\centering
		\begin{tikzpicture}[x=2em,y=-2em]
			\draw (0,0) rectangle +(1,1);
			\draw (1,0) rectangle +(1,1);
			\draw (2,0) rectangle +(1,1);
			\draw (3,0) rectangle +(1,1);
			\draw (0.5,1) rectangle +(1,1);
			\draw (1.5,1) rectangle +(1,1);
			\draw (2.5,1) rectangle +(1,1);
			\draw (3.5,1) rectangle +(1,1);
			\draw (1,2) rectangle +(1,1);
			\draw (2,2) rectangle +(1,1);
			\draw (3,2) rectangle +(1,1);
			\draw (1.5,3) rectangle +(1,1);
			\draw (2.5,3) rectangle +(1,1);
			\draw (3.5,3) rectangle +(1,1);
			\draw (2,4) rectangle +(1,1);
			\draw (3,4) rectangle +(1,1);
			\draw (2.5,5) rectangle +(1,1);
			\draw (3.5,5) rectangle +(1,1);
			\draw (3,6) rectangle +(1,1);
			\draw (3.5,7) rectangle +(1,1);
			\draw (0.5,0.5) node{$\lambda_1$};
			\draw (1.5,0.5) node{$\lambda_2$};
			\draw (2.5,0.5) node{$\cdots$};
			\draw (3.5,0.5) node{$\lambda_n$};
		\end{tikzpicture}
		\caption{Board defining generalized Gelfand--Tsetlin patterns for $\mathfrak{sp}_{2n}$ and $\mathfrak{o}_{2n+1}$.}
		\label{ubf7JTGE}
	\end{figure}
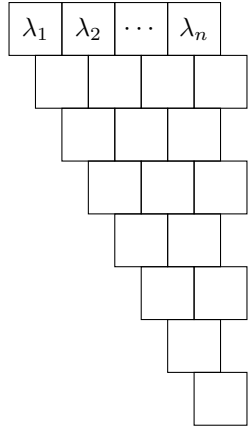	
	They show that $\dim V(\lambda)$ is the number of such patterns \cite[Th.\ 4.2]{berenstein1989}.
	
	Let $\mathcal{O}(P,A)_{(\lambda,0)}$ and $\mathcal{C}(P,A)_{(\lambda,0)}$ be the marked order and chain polytopes determined by the marked Hasse diagram given in \autoref{LCvhqtMX}.
	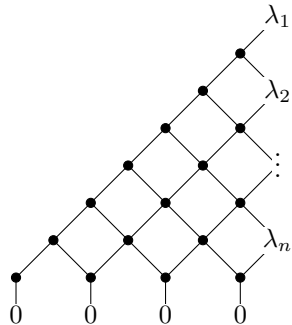
\begin{figure}
		\centering
		\begin{tikzpicture}[x=1.414em,y=1.414em]
			\draw (0,0) -- (7,7);
			\draw (2,0) -- (7,5);
			\draw (4,0) -- (7,3);
			\draw (6,0) -- (7,1);
			\draw (1,1) -- (2,0);
			\draw (2,2) -- (4,0);
			\draw (3,3) -- (6,0);
			\draw (4,4) -- (7,1);
			\draw (5,5) -- (7,3);
			\draw (6,6) -- (7,5);
			\draw (0,0) -- (0,-1);
			\draw (2,0) -- (2,-1);
			\draw (4,0) -- (4,-1);
			\draw (6,0) -- (6,-1);
			\draw (7,7) node[fill=white,inner sep=0.2ex]{$\lambda_1$};
			\draw (7,5) node[fill=white,inner sep=0.2ex]{$\lambda_2$};
			\draw (7,3) node[fill=white,inner sep=0.2ex,text height=2ex]{$\vdots$};
			\draw (7,1) node[fill=white,inner sep=0.2ex]{$\lambda_n$};
			\draw (0,-1) node[fill=white,inner sep=0.2ex]{$0$};
			\draw (2,-1) node[fill=white,inner sep=0.2ex]{$0$};
			\draw (4,-1) node[fill=white,inner sep=0.2ex]{$0$};
			\draw (6,-1) node[fill=white,inner sep=0.2ex]{$0$};
			\filldraw (0,0) circle(0.4ex);
			\filldraw (2,0) circle(0.4ex);
			\filldraw (4,0) circle(0.4ex);
			\filldraw (6,0) circle(0.4ex);
			\filldraw (2,2) circle(0.4ex);
			\filldraw (4,2) circle(0.4ex);
			\filldraw (6,2) circle(0.4ex);
			\filldraw (4,4) circle(0.4ex);
			\filldraw (6,4) circle(0.4ex);
			\filldraw (6,6) circle(0.4ex);
			\filldraw (1,1) circle(0.4ex);
			\filldraw (3,1) circle(0.4ex);
			\filldraw (5,1) circle(0.4ex);
			\filldraw (3,3) circle(0.4ex);
			\filldraw (5,3) circle(0.4ex);
			\filldraw (5,5) circle(0.4ex);
		\end{tikzpicture}
		\caption{Marked Hasse diagram for $\mathfrak{sp}_{2n}$ and $\mathfrak{o}_{2n+1}$.}
		\label{LCvhqtMX}
	\end{figure}
	Note that \autoref{LCvhqtMX} is obtained from \autoref{ubf7JTGE} by a clockwise rotation by $90^\circ$ and apposition of the zeroes.
	From the definitions it is immediate that the $\mathfrak{sp}_{2n}$-patterns of highest weight $\lambda$ are the integral points in $\mathcal{O}(P,A)_{(\lambda,0)}$. This suggests the following:
	
	\begin{conjecture}
	The lattice points in $\mathcal{C}(P,A)_{(\lambda,0)}$ parametrize a PBW basis of $V(\lambda)$ for the symplectic Lie algebras, as described in \autoref{F6Z6kCUH} and in \cite[Theorem 3.11]{feigin2010}.
	\end{conjecture}

	Indeed, this conjecture is proved in an article in preparation by Feigin, Fourier, and Littelmann. \cite{fourier2010}
	
\subsection{\textsf{Type $B$.}}

	For the odd orthogonal Lie algebra $\mathfrak{o}_{2n+1}$, the situation is a bit more complicated.
	Fix a Cartan subalgebra $\mathfrak{h} \subset \mathfrak{o}_{2n+1}$.
	Choose simple roots $\alpha_1, \ldots, \alpha_n \in \mathfrak{h}^*$ such that $\alpha_i \not\perp \alpha_{i + 1}$ for $i < n$ and $\alpha_n$ is the short root.
	Let $\epsilon_1, \ldots, \epsilon_n$ be the basis of $\mathfrak{h}^*$ such that $\alpha_i = \epsilon_i - \epsilon_{i+1}$ for $i < n$ and $\alpha_n = \epsilon_n$.
	The corresponding fundamental weights are $\omega_i = \epsilon_1 + \cdots + \epsilon_i$ for $i < n$ and $\omega_n = \frac 12(\epsilon_1 + \cdots + \epsilon_n)$.
	This is the setting as used by Bourbaki \cite{bourbaki1981}.
	We identify a weight $\lambda$ with the $n$-tuple $(\lambda_1, \ldots, \lambda_n)$ of its coefficients with respect to the basis $\epsilon_1, \ldots, \epsilon_n$.
	Then dominant integral weights correspond to weakly decreasing $n$-tuples in $\frac 12 \mathbf{Z}_{\geq 0}$ such that either all or none of the components are integers.
	Given a dominant integral weight $\lambda$, Berenstein and Zelevinsky \cite{berenstein1989} define an \definedterm{$\mathfrak{o}_{2n+1}$-pattern} of highest weight $\lambda$ to be a filling of the board in \autoref{ubf7JTGE} with elements of $\frac 12 \mathbf{Z}_{\geq 0}$ such that every number is bounded from above by its upper left neighbor and bounded from below by its upper right neighbor (if any), and such that all numbers which possess an upper right neighbor are congruent to $\lambda_1$ modulo $\mathbf{Z}$.
	Let $R(\lambda)$ be the set of $\mathfrak{o}_{2n+1}$-patterns of highest weight $\lambda$.
	
	As in type $C$, let $\mathcal{O}(P,A)_{(\lambda,0)}$ be the marked order polytope defined by the marked Hasse diagram in \autoref{LCvhqtMX}.
	Then $R(\lambda) \subset \mathcal{O}(P,A)_{(\lambda,0)}$, but $R(\lambda)$ does not consist of the integral points, but of the points determined by more complicated congruence conditions.
	Namely, decompose
	\[
		P - A = P' \cup P'' \cup P''' ,
	\]
	where $P'$, $P''$, and $P'''$ consist of all elements in $P$ of height $1$, $2$, and $\geq 3$, respectively, that are not contained in $A$.
	Then $R(\lambda)$ consists of all $x \in \mathcal{O}(P,A)_{(\lambda,0)} \cap (\frac 12 \mathbf{Z})^{P-A}$ such that $x_p + \lambda_1 \in \mathbf{Z}$ for all $p \in P'' \cup P'''$.
	Hence $S(\lambda) = \tilde\phi(R(\lambda))$ consists of all
	\[
		y \in \mathcal{C}(P,A)_{(\lambda,0)} \cap \left( (\tfrac 12 \mathbf{Z})^{P' \cup P''} \times \mathbf{Z}^{P'''} \right)
	\]
	such that
	\[
		\max{} \{ y_q : p \succ q \} + y_p + \lambda_1 \in \mathbf{Z}
	\]
	for all $p \in P''$.
	From the point of view taken in this article, $S(\lambda)$ appears to be the most natural candidate to parametrize a PBW basis of \cite{feigin2010} in type $C$. Note that the elements of $S(\lambda)$ can not appear directly as exponent vectors of a PBW basis, as their components are not necessarily integral, so we are missing at least a change of coordinates in this case.
	\begin{question}
	Is there a way to modify $S(\lambda)$ so that it parametrizes a PBW basis of $V(\lambda)$ for the odd orthogonal Lie algebras, as described in \autoref{F6Z6kCUH} and in \cite[Theorem 3.11]{feigin2010}?
	\end{question}

\subsection{\textsf{Type $D$.}}

	The generalized Gelfand--Tsetlin polytopes \cite{berenstein1989} for the even orthogonal Lie algebras $\mathfrak{o}_{2n}$ are not marked order polytopes, so our methods do not apply here. It would be interesting to find a suitable modification of our results to this case.

\providecommand{\doi}[1]{\href{http://dx.doi.org/#1}{\nolinkurl{doi:#1}}}
\providecommand{\arxiv}[1]{\href{http://arxiv.org/abs/#1}{\nolinkurl{arXiv:#%
1}}}


\begin{thebibliography}{1}

\bibitem{berenstein1989}
A.~D. Berenstein and A.~V. Zelevinsky, \emph{Tensor product multiplicities and
  convex polytopes in partition space}, J. Geom. Phys. \textbf{5} (1989),
  453--472, \doi{10.1016/0393-0440(88)90033-2}.

\bibitem{bourbaki1981}
N.~Bourbaki, \emph{Groupes et alg\`ebres de Lie, chapitres 4, 5 et 6},
  Masson, 1981, \doi{10.1007/978-3-540-34491-9}.

\bibitem{feigin2010}
Evgeny Feigin, Ghislain Fourier, and Peter Littelmann, \emph{PBW filtration
  and bases for irreducible modules in type $A_n$}, \arxiv{1002.0674v3},
  2010.

\bibitem{fourier2010}
Ghislain Fourier, personal communication, 2010.

\bibitem{gelfand1950}
Izrail~M. Gelfand, \emph{Finite-dimensional representations of the group of
  unimodular matrices} (with M. L. Tsetlin), Collected papers,
  vol.~II, Springer, 1988, pp.~653--656.
  Originally appeared in Dokl. Akad. Nauk SSSR \textbf{71} (1950), 825--828.

\bibitem{Stanley86}
Richard~P. Stanley, \emph{Two poset polytopes}, Discrete Comput. Geom.
  \textbf{1} (1986), 9--23, \doi{10.1007/BF02187680}.

\bibitem{Stanleybook}
Richard~P. Stanley, \emph{Enumerative combinatorics}, vol.~1, Cambridge
  University Press, 1997.

\end{thebibliography}
\end{document}